\documentclass[11pt]{amsart}

\usepackage{amssymb,amsfonts}
\usepackage[all,arc]{xy}
\usepackage{enumerate}
\usepackage{mathrsfs}
\usepackage{eufrak}
\usepackage{setspace}
\usepackage{color}
\usepackage{graphicx}
\usepackage{pdfpages}
\usepackage{enumitem}

\newtheorem{theorem}{Theorem}[section]
\newtheorem{proposition}[theorem]{Proposition}

\newtheorem{definition}[theorem]{Definition}

\newtheorem{example}[theorem]{Example}

\theoremstyle{remark}
\newtheorem{remark}[theorem]{Remark}

\renewenvironment{proof}{{\noindent\bf Proof.}}{\hfill $\Box$\par\vskip3mm}

\newcommand{\Aa}{\mathcal{A}}
\newcommand{\Bb}{\mathcal{B}}
\newcommand{\Cc}{\mathcal{C}}

\newcommand{\Ff}{\mathcal{F}}

\def\QQ_{{\mathbb Q}}

\theoremstyle{definition}

\theoremstyle{remark}

\makeatletter
\let\c@equation\c@thm
\makeatother
\numberwithin{equation}{section}

\title{$n$-representations of Quivers}
\author{ADNAN H. ABDULWAHID}
\date{}

\begin{document}

\begin{abstract}
\noindent Let $n \geq 2$. We introduce the notion of $n$-representations of quivers, and we explicitly provide concrete examples of $2$-representations of quivers. We establish the categories of $n$-representations and investigate kernels and cokernels in the categories of $n$-representations of quivers. Further, we construct them in terms of kernels and cokernels of morphisms in the usual categories of quiver representations. We show that every morphism in the categories of  $n$-representations has a canonical decomposition.  Most importantly, we prove that the categories of $n$-representations of quivers are  $k$-linear abelian categories.  \\   
\end{abstract}

\date{}

\maketitle


\noindent

\section{Introduction}
\label{intro}
The notions of quiver and their representation can be traced back to 1972 when they were introduced by Gabriel \cite{Gabriel}. Since then, it has been studied as a vibrant subject with a strong linkage with many other mathematics areas. This comes from the modern approach that quiver representations theory suggests. Due to its inherent combinatorial flavor, this theory has recently been largely studied as extremely important theory with connections to many theories, such as associative algebra, combinatorics, algebraic topology, algebraic geometry, quantum groups, Hopf algebras, tensor categories. Further, it bridges the gap between combinatorics and category theory, and this simply comes from the well-known fact that there is a forgetful functor, which has a left adjoint,  from the category of small categories to the category of quivers. It turns out that it gives``new techniques, both of combinatorial, geometrical and categorical nature." \cite[p. ix]{Buan}.\\

The interaction area of quivers representation theory with other branches of mathematics can be significantly extended by introducing a generalization of this theory. However, suggesting a useful generalization needs to be done carefully because not all  generalizations are capable of supporting our goal of finding a generalization that plays a successful role in developing this theory.\\

Furthermore, generalizing the notion of an object (or objects) with  structures can be done with no compatibility condition between these structures, or with a compatibility condition between them. For instance, bitopological spaces can be regarded as a generalization of the notion topological spaces. A bitopological space, introduced by Kelly in \cite{Kelly}, is a triple $(X,\tau,\tau')$, where  $X$ is a set equipped with two arbitrary topologies $\tau, \, \tau'$ \cite[p. ix]{Dvalishvili}. Obviously, this definition does not require any compatibility condition between $\tau, \, \tau'$. However, it is still very important, and indeed chapter $VII$ in \cite[p. 318-384]{Dvalishvili}  is totally devoted for  applications of bitopologies.\\ 

On the other hand, there is an another kind of generalization involved with compatibility condition. For example, the concept of corings is a  generalization of that of coalgebras, and it involves certain compatibility conditions. The compatibility conditions are substantially helpful in characterizing and describing  many notions. \\

The notion of $n$-representations of quivers can be introduced as a generalization with a compatibility condition. We start with  $2$-representations of quivers and inductively define $n$-representations quivers. Then we mainly concentrate our study on $2$-representations of quivers because they roughly give a complete description of $n$-representations of quivers which can be established inductively. We alternatively and preferably call $2$-representations of quivers  birepresentations of quivers.\\

Birepresentations of quivers are fundamentally different from representations of biquivers\footnote{A directed graph with usual and dashed arrows will be called a \textbf{biquiver}. Its \textbf{representation} is given by assigning to each vertex a complex vector space, to each usual arrow a linear mapping, and to each dashed arrow a semilinear mapping \cite[p. 237]{Sergeichuk}.} introduced by Sergeichuk in \cite[p. 237]{Sergeichuk}.\\

The main goal of this paper is to introduce the concept of $n$-representations of quivers and set up the basic notions of this concept. Further, we mainly establish the categories of $n$-representations of quivers and show that these categories are abelian. \\
 
As a part of our next paper, we will show that $n$-representations of quivers can be identified as representations of certain quivers. We will intentionally not use this observation in this paper since this allows us to explore a more explicit description and characterization for $n$-representations of quivers without using the categorical perspective description of being ``essentially the same".
  
The sections of this paper can be summarized in the following setting.\\
In Section 2, we give some detailed background on quiver representations and few categorical notions that we need for the next sections. \\
In Section 3, we introduce the notion of $n$-representations of quivers, we explicitly give concrete examples of birepresentations of quivers. In addition, we establish the categories of 
$n$-representations of quivers.  \\
In Section 4, we investigate the kernels and cokernels in the categories of $n$-representations of quivers. We also construct them in terms of kernels and cokernels in the usual categories of quiver representations corresponding to each component. \\
In Section 5, we show that the morphisms in the categories of $n$-representations of quivers have canonical decomposition. We also show that each hom set in these categories is equipped with a structure of an abelian group such that composition of morphisms is biadditive with respect to this structure. We end the paper by showing that   the categories of $n$-representations of quivers are abelian.\\

\section{\textbf{Preliminaries}}\label{s.p}
Throughout this paper $k$ is an algebraically closed field, $n \geq 2$,  and $Q, \,\, Q', \,\, Q_1, \,\, Q_2, ..., \,\, Q_n$ are  quivers. We also denote $kQ$ the path algebra of $Q$. Unless otherwise specified, we will consider only finite, connected, and acyclic quivers. \\
Let $\mathcal{A}$ be a (locally small) category and $A$, $B$  objects in $\mathcal{A}$. We denote by $\mathcal{A} (A,B)$ the set of all morphisms from $A$ to $B$. \\

Let $\Aa$, $\Bb$ be categories. Following \cite[p. 74]{McLarty},  the  \textbf{product category }$\Aa \times \Bb$ is the category whose objects are all pairs of the form $(A,B)$, where $A$ is an object of $\Aa$ and $B$ an object of $\Bb$. An arrow is a pair $(f,g): (A,B) \rightarrow (A',B')$, where $f: A \rightarrow A'$ is an arrow of $\Aa$ and $g: B \rightarrow B'$ is an arrow of $\Bb$. The identity arrow for $\Aa \times \Bb$ is $(id_{\!_{A}},id_{\!_{B}})$ and composition is defined component-wise, so $(f,g)(f',g') = (ff',gg')$. There is a projective functor $P_1:\Aa \times \Bb \rightarrow \Aa$ defined by $P_1(A,B)= A $ and $P_1 (f,g) =f$. Similarly, we have a projective functor  $P_2:\Aa \times \Bb \rightarrow \Bb$ defined by $P_2(A,B)= B $ and $P_2 (f,g) = g$.  For the fundamental concepts of category theory, we refer to \cite{Leinster1}, \cite{Mac Lane1}, \cite{Awodey}, \cite{Rotman}, \cite{Adamek}, \cite{Borceux1}, \cite{Freyd}, \cite{Pareigis}, or  \cite{Mitchell}.\\

Following \cite{Schiffler}, a \textbf{quiver} $Q = (\mathsf{Q}_{\!_{0}}, \mathsf{Q}_{\!_{1}}, s,t)$  consists of
\begin{itemize}
\item $\mathsf{Q}_{\!_{0}}$ a set of vertices,
\item $\mathsf{Q}_{\!_{1}}$ a set of arrows,
\item $s:\mathsf{Q}_{\!_{1}} \rightarrow \mathsf{Q}_{\!_{0}}$ a map from arrows to vertices, mapping an arrow to its starting
point,
\item $t:\mathsf{Q}_{\!_{1}} \rightarrow \mathsf{Q}_{\!_{0}}$ a map from arrows to vertices, mapping an arrow to its terminal point.
\end{itemize}
We will represent an element $\alpha \in \mathsf{Q}_{\!_{1}}$ by drawing an arrow from its starting point $s(\alpha)$ to its endpoint $t(\alpha)$ as follows:
$s(\alpha) \xrightarrow{\alpha} t(\alpha)$.\\
A \textbf{representation} $M = (M_i,\varphi_{\alpha})_{i \in \mathsf{Q}_{\!_{0}}, \alpha \in \mathsf{Q}_{\!_{1}}}$ of a quiver $Q$ is a collection of $k$-vector spaces $M_i$ one for each vertex $i \in \mathsf{Q}_{\!_{0}}$, and a collection of $k$-linear maps $\varphi_{\alpha}: M_{s(\alpha)} \rightarrow M_{t(\alpha)}$ one for each arrow $\alpha \in \mathsf{Q}_{\!_{1}}$. \\
A representation $M$ is called \textbf{finite-dimensional} if each vector space Mi is finite-dimensional.\\
Let $Q$ be a quiver and let $M = (M_i,\varphi_{\alpha})$, $M' = (M'_i,\varphi'_{\alpha})$  be two representations of $Q$. A \textbf{morphism} of representations $f : M \rightarrow M'$ is a collection $(f_i)_{i \in \mathsf{Q}_{\!_{0}}}$ of $k$-linear maps $f_i : M_i \rightarrow M'_i$
such that for each arrow $s(\alpha) \xrightarrow{\alpha} t(\alpha)$ in 
 $\mathsf{Q}_{\!_{1}}$ the diagram 
 \begin{equation} \label{diag.eq01}
\xymatrix{
M_{s(\alpha)} \ar[rr]^{\phi_{\alpha}} \ar[d]_{f_{s(\alpha)}} && M_{t(\alpha)} \ar[d]^{f_{t(\alpha)}}\\
M'_{s(\alpha)} \ar[rr]_(.5){\phi'_{\alpha}} && M'_{t(\alpha)}
} 
\end{equation}
commutes. \\
A morphism of representations $f = (f_i): M \rightarrow M'$  is an isomorphism if each $f_i$ is bijective. The class of all representations that are isomorphic to a given representation $M$ is called the \textbf{isoclass} of $M$.\\
This gives rise to define a category $Rep_k(Q)$ of $k$-linear representations of  $Q$. We denote by $rep_k(Q)$ the full  subcategory of $Rep_k(Q)$ consisting of the finite dimensional representations.

Given two representations $M = (M_i,\phi_{\alpha})$ and  $M' = (M'_i,\phi'_{\alpha'})$ of $Q$, the representation

\begin{equation} \label{diag.eq001}
M \oplus M' \,\,\ = \,\,\ (M_i \oplus M'_i,\begin{bmatrix}
       \phi_{\alpha} & 0            \\[0.3em]
       0 & \phi'_{\alpha}         \\[0.3em] \end{bmatrix})
\end{equation}
is the \textbf{direct sum} of $M$ and $M'$ in $Rep_k(Q)$ \cite[p. 71]{Assem}. \\
A nonzero representation of a quiver $Q$ is said to be \textbf{indecomposable} if it is not isomorphic to a direct sum of two nonzero representations \cite[p. 21]{Etingof1}. \\

We will need the following propositions.

\begin{proposition} \label{p.1} \cite[p. 70]{Assem}
Let $Q$ be a finite quiver. Then $Rep_k(Q)$ and $rep_k(Q)$
are $k$-linear abelian categories.
\end{proposition}

\begin{proposition} \label{p.Equiv.Rep.Mod} \cite[p. 74]{Assem}
Let $Q$ be a finite, connected, and acyclic quiver. There
exists an equivalence of categories $Mod \, kQ \simeq Rep_k(Q)$ that restricts to an equivalence $mod \, kQ \simeq rep_k(Q)$, where $kQ$ is the path algebra of $Q$, $Mod \, kQ$ denotes the category of right $kQ$-modules, and  $mod \, kQ$ denotes the full subcategory of $Mod \, kQ$ consisting of the finitely generated right $kQ$-modules.
\end{proposition}  

This is a very brief review of the basic concepts involved with our work. For the basic notions of quiver representations theory, we refer the reader to \cite{Assem}, \cite{Schiffler}, \cite{Auslander1}, \cite{Barot}, \cite{Etingof1}, \cite{Benson}, \cite{Zimmermann}.\\

\vspace{.2cm}

\section{\textbf{$n$-representations of Quivers: Basic Concepts}}\label{s.n.rep.quivers}
Let $Q = (\mathsf{Q}_{\!_{0}}, \mathsf{Q}_{\!_{1}}, s,t)$, $Q' = (\mathsf{Q}'_{\!_{0}}, \mathsf{Q}'_{\!_{1}}, s',t')$ be quivers. 

\begin{definition} \label{def.1} 
A \textbf{$2$-representation} of  $(Q,Q')$  (or a \textbf{birepresentation} of $(Q,Q')$) is a triple $\bar{M} = ((M_i,\phi_{\alpha}),(M'_{i'},\phi'_{\beta}), (\psi^{\alpha}_{\beta}))_{i \in \mathsf{Q}_{\!_{0}}, i' \in Q'_{\!_{0}}, \alpha \in \mathsf{Q}_{\!_{1}},\beta \in \mathsf{Q}'_{\!_{1}}}$, where $(M_i,\phi_{\alpha}),(M'_{i'},\phi'_{\beta})$ are representations of $Q,Q'$  respectively, and $(\psi^{\alpha}_{\beta})$ is a collection of $k$-linear maps $\psi^{\alpha}_{\beta}: M_{t(\alpha)} \rightarrow M'_{s(\beta)} $,  one for each pair of arrows $(\alpha,\beta) \in \mathsf{Q}_{\!_{1}} \times  \mathsf{Q}'_{\!_{1}}$. \\ 

Unless confusion is possible, we denote a birepresentation simply by $\bar{M} = (M,M',\psi)$. Next, we inductively define $n$-representations for any integer $n \geq 2$. \\
For any  $m \in \{2,\,\,..., \,\,n\}$, let $Q_m = (\mathsf{Q}^{(m)}_{\!_{0}}, \mathsf{Q}^{(m)}_{\!_{1}}, s^{(m)},t^{(m)})$ be a quiver.  A \textbf{$n$-representation} of $(Q_1, \,\, Q_2, \,\,... \,\, ,  Q_n)$ is $(2n-1)$-tuple  $\bar{V} = (V^{(1)},V^{(2)},...,V^{(n)},\psi_{\!_{1}}, \psi_{\!_{2}}, ..., \psi_{\!_{n-1}})$, where for every $m \in \{1, \,\,2,\,\,..., \,\,n\}$, $V^{(m)}$ is a representation of  $Q_m$, and $(\psi_{\!_{m}\gamma^{(m-1)}}^{\gamma^{(m)}}) $ is a collection of $k$-linear maps 
\begin{center}
•$\psi_{\!_{m}\gamma^{(m-1)}}^{\gamma^{(m)}}: V^{(m-1)}_{t^{(m-1)}(\gamma^{(m-1)})} \rightarrow V^{(m-1)}_{s^{(m)}(\gamma^{(m)})} $,
\end{center}
  one for each pair of arrows $(\gamma^{(m-1)},\gamma^{(m)}) \in \mathsf{Q}^{(m-1)}_{\!_{1}} \times  \mathsf{Q}^{(m)}_{\!_{1}}$  and $m \in \{2,\,\,..., \,\,n\}$.   \\ 

\end{definition}

\begin{remark} \label{r.inductive} \textbf{•}
\begin{enumerate}[label=(\roman*)]
\item When no confusion is possible, we simply write $s,t$ instead of $s',t'$ respectively, and for every  $m \in \{1, \,\,2,\,\,..., \,\,n\}$, we write $s,t$ instead of $s^{(m)},\,t^{(m)}$ respectively. 
\item It is clear that if $(V^{(1)},V^{(2)},...,V^{(n)},\psi_{\!_{1}}, \psi_{\!_{2}}, ..., \psi_{\!_{n-1}})$ is an $n$-representation of $(Q_1, \,\, Q_2, \,\,... \,\, ,  Q_n)$, then  $(V^{(1)},V^{(2)},...,V^{(n-1)},\psi_{\!_{1}}, \psi_{\!_{2}}, ..., \psi_{\!_{n-2}})$ is an $(n-1)$-representation of $(Q_1, \,\, Q_2, \,\,... \,\, ,  Q_{n-1})$ for every integer $n \geq 2$.
\item Part $(ii)$ implies that for any integer $n > 2$, $n$-representations roughly inherit all the properties and the universal constructions that  $(n-1)$-representations have. Thus, we mostly focus on studying birepresentations since they can be regarded as a mirror in which one can see a clear decription of $n$-representations for any integer $n > 2$.\\
\end{enumerate} 

\end{remark}

\begin{example} \label{ex.1} 
Let $Q,Q'$ be the following quivers 
\begin{equation} \label{diag.eq06}
\xymatrix{
&&\\
Q: & 1 \ar[r]^{} & 2 
}
\hspace{80pt}
\xymatrix{
&1 \ar[dr]^{} & &\\
Q':&&3  &4 \ar[l]^{}\\
&2  \ar[ur]^{} &&
}
\end{equation}
and consider the following:
\begin{equation} \label{diag.eq07}
\xymatrix{
&k \ar[dr]^{\begin{bmatrix}
       1            \\[0.3em]
       0         \\[0.3em]
       \end{bmatrix}} & &\\
M &&k^2  &k \ar[l]_{\begin{bmatrix}
       1            \\[0.3em]
       1         \\[0.3em]
       \end{bmatrix}}\\
&k  \ar[ur]_{\begin{bmatrix}
       0            \\[0.3em]
       1         \\[0.3em]
       \end{bmatrix}} &&
}
\hspace{70pt}
\xymatrix{
&&&\\
M' &k  &k \ar[l]_{1}
}
\end{equation}
Then $M$ (respectively $M'$) is a representation of $Q$ ((respectively $Q'$) \cite{Schiffler}. The following are birepresentations of $(Q,Q')$.
\begin{equation} \label{diag.eq08}
\xymatrix{
&k \ar[dr]^{\begin{bmatrix}
       1            \\[0.3em]
       0         \\[0.3em]
       \end{bmatrix}} & &\\
\bar{M} &&k^2  &k \ar[l]_{\begin{bmatrix}
       1            \\[0.3em]
       1         \\[0.3em]
       \end{bmatrix}} & &k \ar[ll]_{1} \ar@/_5pc/[llllu]_{1} 
       \ar@/^5pc/[lllld]^{1} &k \ar[l]_{1}\\
&k  \ar[ur]_{\begin{bmatrix}
       0            \\[0.3em]
       1         \\[0.3em]
       \end{bmatrix}} &&
}
\end{equation}
\begin{equation} \label{diag.eq09}
\xymatrix{
&k \ar[dr]^{\begin{bmatrix}
       1            \\[0.3em]
       0         \\[0.3em]
       \end{bmatrix}} & &\\
\bar{N} &&k^2  &k \ar[l]_{\begin{bmatrix}
       1            \\[0.3em]
       1         \\[0.3em]
       \end{bmatrix}} & &k \ar[ll]_{1} \ar@/_5pc/[llllu]_{0} 
       \ar@/^5pc/[lllld]^{0} &k \ar[l]_{1}\\
&k  \ar[ur]_{\begin{bmatrix}
       0            \\[0.3em]
       1         \\[0.3em]
       \end{bmatrix}} &&
}
\end{equation}
\end{example} 

\vspace{.1cm}

\begin{definition} \label{def.2} 
Let $\bar{V} =(V,V',\psi), \,\, \bar{W} =(W,W',\psi')$ be birepresentations of $(Q,Q')$. Write 
$ V = (V_i,\phi_{\alpha})$, $V'=(V'_{i'},\mu_{\beta})$, 
$ W = (W_i,\phi'_{\alpha})$, $W'=,(W'_{i'},\mu'_{\beta})$. A morphism of birepresentations  $ \bar{f}: \bar{V} \rightarrow \bar{W}$ is a pair  $\bar{f} = (f,f')$, where $f=(f_i):  (V_i,\phi_{\alpha}) \rightarrow (W_i,\phi'_{\alpha})$, $f'=(f'_{i'}):  (V'_{i'},\mu_{\beta}) \rightarrow (W'_{i'},\mu'_{\beta})$ are morphisms in  $Rep_k(Q)$,  $Rep_k(Q')$  respectively such that the following diagram commutes. 
\begin{equation} \label{diag.eq010}
\xymatrix{
V_{s(\alpha)} \ar[rr]^{\phi_{\alpha}} \ar[dr]_{f_{s(\alpha)}} && V_{t(\alpha)} \ar[dd]|\hole^(.3){\psi^{\alpha}_{\beta}} \ar[dr]^{f_{t(\alpha)}}\\
& W_{s(\alpha)} \ar[rr]|(.3){\phi'_{\alpha}} && W_{t(\alpha)} \ar[dd]^(.32){\psi'^{\alpha}_{\beta}}\\
&& V'_{s(\beta)} \ar[rr]|\hole|(.65){\mu_{\beta}} \ar[dr]_{f'_{s(\beta)}} &&
V'_{t(\beta}  \ar[dr]^{f'_{t(\beta)}} \\
&&&  W'_{s(\beta)}  \ar[rr]_{\mu'_{\beta}} &&  W'_{t(\beta)}
}
\end{equation}
The composition of two maps $(f,f')$ and $(g,g')$ can be depicted as the following diagram.
\begin{equation} \label{def.eq011}
\xymatrix{
V_{s(\alpha)} \ar[rr]^{\phi_{\alpha}} \ar[dr]_{f_{s(\alpha)}} && V_{t(\alpha)} \ar[ddd]|!{[dd];[d]}\hole|!{[ddd];[dd]}\hole^(.5){\psi^{\alpha}_{\beta}}  \ar[dr]^{f_{t(\alpha)}}\\
& W_{s(\alpha)} \ar[rr]|(.3){\phi'_{\alpha}}  \ar[dr]_{g_{s(\alpha)}}
&& W_{t(\alpha)} \ar[ddd]|!{[dd];[d]}\hole^(.5){\psi'^{\alpha}_{\beta}}  \ar[dr]^{g_{t(\alpha)}}\\
&& U_{s(\alpha)} \ar[rr]|(.3){\phi''_{\alpha}} && U_{t(\alpha)} \ar[ddd]|(.5){\psi''^{\alpha}_{\beta}} \\
&& V'_{s(\beta)} \ar[rr]|\hole|(.65){\mu_{\beta}} \ar[dr]_{f'_{s(\beta)}} &&
V'_{t(\beta}  \ar[dr]^{f'_{t(\beta)}} \\
&&&  W'_{s(\beta)}  \ar[dr]_{g'_{s(\beta)}} \ar[rr]|\hole^(.3){\mu'_{\beta}} &&  W'_{t(\beta)}\ar[dr]^{g'_{t(\beta)}}\\
&&&&  U'_{s(\beta)}  \ar[rr]_(.5){\mu''_{\beta}} &&  U'_{t(\beta)}
}
\end{equation}
In general, if $\bar{V} = (V^{(1)},V^{(2)},...,V^{(n)},\psi_{\!_{1}}, \psi_{\!_{2}}, ..., \psi_{\!_{n-1}})$, $\bar{W} = (W^{(1)},W^{(2)},...,W^{(n)},\psi'_{\!_{1}}, \psi'_{\!_{2}}, ..., \psi'_{\!_{n-1}})$ are  $n$-representations of $(Q_1, \,\, Q_2, \,\,... \,\, ,  Q_n)$, then a morphism of $n$-representations 
$\bar{f}: \bar{V} \rightarrow \bar{W}$ is $n$-tuple $\bar{f}=(f^{\!^{(1)}}, f^{\!^{(2)}}, ..., f^{\!^{(n-1)}})$, where 

\begin{center}
•$f^{\!^{(m)}} = (f^{\!^{(m)}}_{i^{(m)}}):  (V_{{i^{(m)}}},\phi^{{i^{(m)}}}_{\gamma^{(m)}}) \rightarrow  (W_{{i^{(m)}}},\mu^{{i^{(m)}}}_{\gamma^{(m)}})$,
\end{center}
 is a morphism in  $Rep_k(Q_m)$ for any $m \in \{2,\,\,..., \,\,n\}$, and for each pair of arrows $(\gamma^{(m-1)},\gamma^{(m)}) \in \mathsf{Q}^{(m-1)}_{\!_{1}} \times  \mathsf{Q}^{(m)}_{\!_{1}}$ the following diagram is commutative.
 
\begin{equation} \label{diag.eq010.1}
\xymatrix{ 
	V^{(m-1)}_{t^{(m-1)}(\gamma^{(m-1)})} \ar[rrr]^{\psi_{\!_{m}\gamma^{(m-1)}}^{\gamma^{(m)}}} \ar[dd]_{f^{(m-1)}_{t^{(m-1)}(\gamma^{(m-1)})}} &&& V^{(m)}_{s^{(m)}(\gamma^{(m)})} \ar[dd]^{f^{(m)}_{s^{(m)}(\gamma^{(m)})}} \\
	&&&\\
	W^{(m-1)}_{t^{(m-1)}(\gamma^{(m-1)})} \ar[rrr]_{{\psi'}_{\!_{m}\gamma^{(m-1)}}^{\gamma^{(m)}}} &&& W^{(m)}_{s^{(m)}(\gamma^{(m)})} 
}
\end{equation}
for every $m \in \{2,\,\,..., \,\,n\}$.

\vspace{.2cm}

A morphism of $n$-representations can be depicted as:

\begin{equation} \label{diag.eq010.2}
\xymatrix{ 
	&  \ar[rr]&&  \ar[dd]&& \ar[rr]\ar[dd]|\hole &&  \ar[dd] &&  \ar[rr] \ar[dd] &&  \ar[dd] &\\
	 \ar[ur] \ar[rr]    &&  \ar[dd] \ar[ur]&& \ar[rr]\ar[dd] \ar[ur]&&  \ar[dd] \ar[ur]&&   \ar[rr] \ar[ur] \ar[dd]&&   \ar[dd] \ar[ur]&\\
	&&& \ar[rr]|\hole&&   &&  \ar[rr]|\hole&&   && \ar@{--}[r]  &\\
	&& \ar[rr]\ar[ur]&&  \ar[ur]&&  \ar[rr]\ar[ur]&&  \ar[ur]  & &   \ar[ur] \ar@{--}[rr] &&
}
\end{equation}

\end{definition}

\vspace{.2cm}

\begin{remark} \label{r.1} \textbf{•}

\begin{enumerate}[label=(\roman*)]
\item The above definition gives rise to form a category  $Rep_{\!_{(Q,Q')}}$ of $k$-linear birepresentations of  $(Q,Q')$. We denote by  $rep_{\!_{(Q,Q')}}$ the full subcategory of $Rep_{\!_{(Q,Q')}}$  consisting of the finite dimensional birepresentations. Similarly, it also creates a category $Rep_{\!_{(Q_1,Q_2,...,Q_n)}}$ of $n$-representations. We denote $rep_{\!_{(Q_1,Q_2,...,Q_n)}}$ the full subcategory of $Rep_{\!_{(Q_1,Q_2,...,Q_n)}}$ consisting of the finite dimensional  $n$-representations.

\item For any  $m \in \{2,\,\,..., \,\,n\}$, let $Q_m = (\mathsf{Q}^{(m)}_{\!_{0}}, \mathsf{Q}^{(m)}_{\!_{1}}, s^{(m)},t^{(m)})$ be a quiver and fix $j \in \{2,\,\,..., \,\,n\}$. Let $\Upsilon_{Rep_k(Q_j)}$ be the subcategory of $Rep_{\!_{(Q_1,Q_2,...,Q_n)}}$ whose objects are $(2n-1)$-tuples  $\bar{X} = (0,0,...,V^{(j)},0,...,0,\psi_{\!_{1}}, \psi_{\!_{2}}, ..., \psi_{\!_{n-1}})$, where $V^{(j)}$ is a representation of $Q_j$,  and $\psi_{\!_{m}\gamma^{(m-1)}}^{\gamma^{(m)}} = 0$ for every pair of arrows $(\gamma^{(m-1)},\gamma^{(m)}) \in \mathsf{Q}^{(m-1)}_{\!_{1}} \times  \mathsf{Q}^{(m)}_{\!_{1}}$ and  $m \in \{2,\,\,..., \,\,n\}$. Then $\Upsilon_{Rep_k(Q_j)}$ is clearly a full subcategory of  $Rep_{\!_{(Q_1,Q_2,...,Q_n)}}$. Notably, we have an equivalence of categories $\Upsilon_{Rep_k(Q_j)} \simeq Rep_k(Q_j)$, and thus by by Proposition (\ref{p.Equiv.Rep.Mod}), we have $Rep_k(Q_j) \simeq \Upsilon_{Rep_k(Q_j)} \simeq Mod \, kQ_j $. It turns out that the category  $Rep_k(Q_j)$ and $ Mod \, kQ_j)$ can be identified as full subcategories of $Rep_{\!_{(Q_1,Q_2,...,Q_n)}}$. \\
The category $\Upsilon_{Rep_k(Q_j)}$ has a full subcategory  $\Upsilon_{rep_k(Q_j)}$ when we restrict the objects on the finite dimensional representations. Therefore, we also have  $rep_k(Q_j) \simeq \Upsilon_{rep_k(Q_j)} \simeq mod \, kQ_j $. \\
\end{enumerate}
\end{remark}

\begin{remark} \label{r.bicategory}  \,\, Let $\Bb_0$ be the class of all quivers. One might consider the class $\Bb_0$ and full subcategories of the categories of birepresentations of quivers to build a bicategory. Indeed, there is a bicategory $\Bb$ consists of 
\begin{itemize}
\item the objects or the $0$-cells of $\Bb$ are simply the elements of  $\Bb_0$
 
\item for each $Q, Q' \in \Bb_0$, we have  $\Bb(Q,Q') =  Rep_k(Q) \times Rep_k(Q')$, whose objects are the $1$-cells of $\Bb$, and whose   morphisms are the $2$-cells of $\Bb$

\item for each $Q, Q', Q'' \in \Bb_0$, \,\, a composition functor 
\begin{center}
•$\Ff: Rep_k(Q') \times Rep_k(Q'') \,\,\, \times \,\,\, Rep_k(Q) \times Rep_k(Q') \rightarrow Rep_k(Q) \times Rep_k(Q'')$
\end{center}
defined by:\\
\begin{center}
•$\Ff ((N',N''),(M,M')) = (M,N''), \,\,\,\,\, \Ff ((g',g''),(f,f')) = (f,g'')$
on $1$-cells $(M,M'), (N',N'')$ and $2$-cells $(f,f'), (g',g'')$.
\end{center}
\item for any  $Q \in \Bb_0$ and for each $(M,M') \in \Bb(Q,Q)$, we have $\Ff ((M,M'),(M,M)) = (M,M')$ \,\, and \,\, $\Ff ((M',M'),(M,M')) = (M,M')$. Furthermore, for any $2$-cell $(f,f')$, we have $\Ff ((f',f'),(f,f')) = (f,f')$  \,\, and \,\, $\Ff ((f,f'),(f,f)) = (f,f')$. Thus, the identity and the unit coherence axioms hold. 
\end{itemize}

The rest of bicategories axioms are obviously satisfied. \,\, For each $Q, Q' \in \Bb_0$, let $\beth_{\!_{(Q,Q')}}$ be the full subcategory of $Rep_{\!_{(Q,Q')}}$ whose objects are the triples $(X,X',\Psi)$, where $(X,X') \in Rep_k(Q) \times Rep_k(Q')$ and $\Psi^{\alpha}_{\beta} = 0$ for every pair of arrows $(\alpha,\beta) \in \mathsf{Q}_{\!_{1}} \times  \mathsf{Q}'_{\!_{1}}$, and whose morphisms are usual morphisms of birepresentations between them. \,\, Clearly, \,\, $\beth_{\!_{(Q,Q')}} \cong Rep_k(Q) \times Rep_k(Q')$  for any $Q, Q' \in \Bb_0$. Thus,  by considering the class $\Bb_0$ and the full subcategories described above of the birepresentations categories of quivers, we can always build a bicategory as above.\\
Obviously, \,\, the discussion above implies that for each $Q, Q' \in \Bb_0$, the product category  $Rep_k(Q) \times Rep_k(Q')$ can be viewed as a full subcategory of $Rep_{\!_{(Q,Q')}}$. Further, it implies that the product category $Rep_k(Q_1) \times Rep_k(Q_2) \times ...\times  Rep_k(Q_n) $ can be viewed as a full subcategory of $Rep_{\!_{(Q_1,Q_2,...,Q_n)}}$, where $Q_1,Q_2,...,Q_n \in \Bb_0$ and $n \geq 2$.\\
We also have the same analogue if we replace $Rep_{\!_{(Q_1,Q_2,...,Q_n)}}$,  by $rep_{\!_{(Q_1,Q_2,...,Q_n)}}$, and $Rep_k(Q_1)$, $Rep_k(Q_2)$, ... ,  $Rep_k(Q_n) $ by  $rep_k(Q_1), rep_k(Q_2)$,  ... , $rep_k(Q_n) $  respectively. For the basic notions of bicategories, we refer the reader to \cite{Leinster2}. \\
\end{remark}

\begin{example} \label{ex.2} 
Let $Q, \, Q'$ be the quivers defined in Example \ref{ex.1} and consider the following:\\

\begin{equation} \label{diag.eq012}
\xymatrix{
&k \ar[dr]^{\begin{bmatrix}
       1            \\[0.3em]
       0         \\[0.3em]
       \end{bmatrix}} & &\\
V &&k^2  &k \ar[l]_{\begin{bmatrix}
       1            \\[0.3em]
       1         \\[0.3em]
       \end{bmatrix}}\\
&k  \ar[ur]_{\begin{bmatrix}
       0            \\[0.3em]
       1         \\[0.3em]
       \end{bmatrix}} &&
}
\hspace{70pt}
\xymatrix{
&k \ar[dr]^{1} & &\\
W &&k  &k \ar[l]_{1}\\
&k  \ar[ur]_{1} &&
}
\end{equation}

\begin{equation} \label{diag.eq013}
\xymatrix{
V' &k  &k \ar[l]_{1}  
}
\hspace{70pt}
\xymatrix{
W'&k  &k \ar[l]_{1}  
}
\end{equation}

\vspace{.2cm}

Then $V, \, V'$ (respectively  $W, \, W'$) are representations of $Q$ (respectively $Q'$)  \cite{Schiffler}. Furthermore, it is straightforward to verify that $Rep_k(Q)(V,W) \cong k^2$ and  $Rep_k(Q)(V',W') \cong k$. We refer the reader to \cite{Schiffler} for more details. Consider the following.
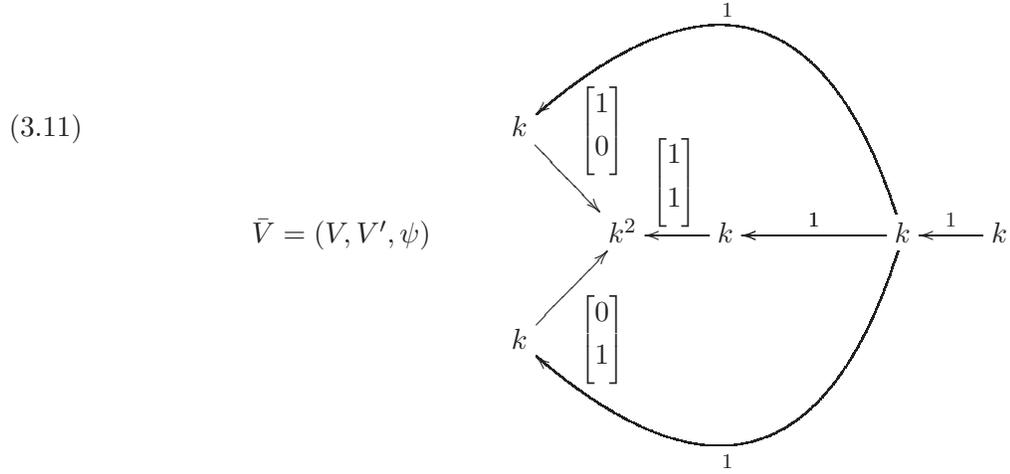
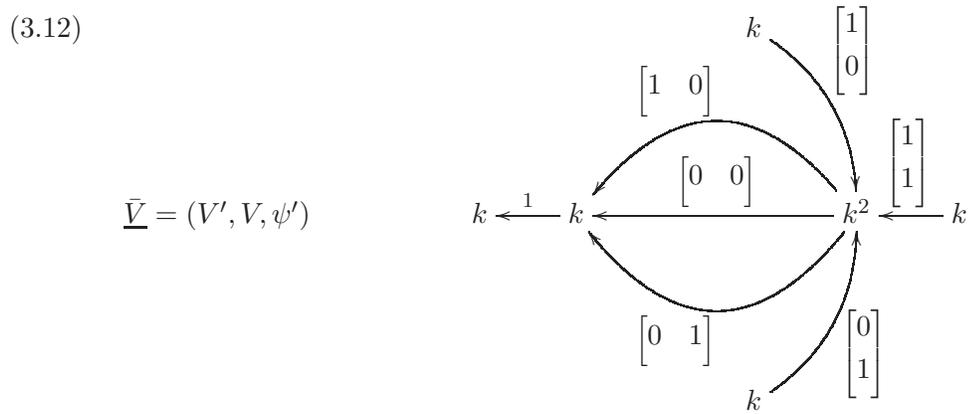
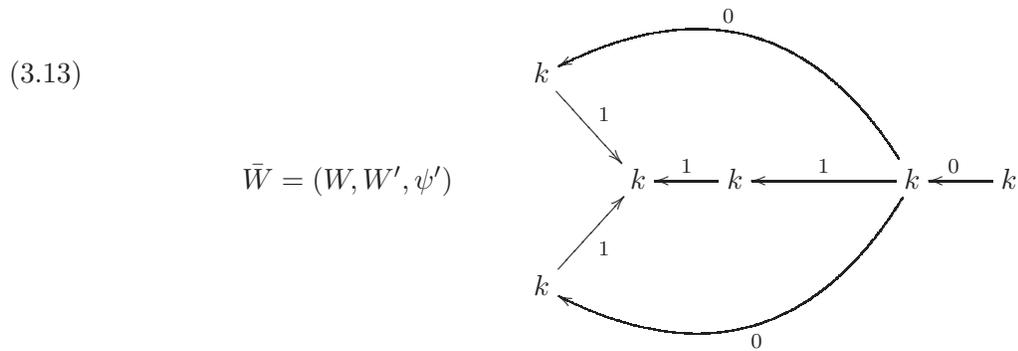
\begin{equation} \label{diag.eq014}
\xymatrix{
&k  \ar[dr]^{\begin{bmatrix}
       1            \\[0.3em]
       0         \\[0.3em]
       \end{bmatrix}} & &\\
\bar{V} = (V,V',\psi) &&k^2 &k \ar[l]_{\begin{bmatrix}
       1            \\[0.3em]
       1         \\[0.3em]
       \end{bmatrix}} & &k \ar[ll]_{1} \ar@/_5pc/[llllu]_(.57){1}  \ar@/^5pc/[lllld]^(.57){1} \ar[ll]_{1} 
        &k \ar[l]_{1} \\
&k  \ar[ur]_{\begin{bmatrix}
       0            \\[0.3em]
       1         \\[0.3em]
       \end{bmatrix}} &&&&&
}
\end{equation}

\begin{equation} \label{diag.eq014.1}
\xymatrix{
&&&&&k  \ar@/^1.pc/[ddr]^(.4){\begin{bmatrix}
       1            \\[0.3em]
       0         \\[0.3em]
       \end{bmatrix}} & &\\
       &&&&&&&\\
\underline{\bar{V}} = (V',V,\psi')  && k    
        &k \ar[l]_{1}  &&&k^2 \ar@/_3pc/[lll]_(.6){\begin{bmatrix}
       1 & 0       \\[0.3em]
              \end{bmatrix}} 
              \ar@/^3pc/[lll]^(.6){\begin{bmatrix}
       0 & 1       \\[0.3em]
              \end{bmatrix}}
               \ar[lll]_(.5){\begin{bmatrix}
       0 & 0       \\[0.3em]
              \end{bmatrix}}
              &k \ar[l]_{\begin{bmatrix}
       1            \\[0.3em]
       1         \\[0.3em]
       \end{bmatrix}} & &\\
         &&&&&&&\\
&&&&&k  \ar@/_1.pc/[uur]_{\begin{bmatrix}
       0            \\[0.3em]
       1         \\[0.3em]
       \end{bmatrix}} &&&&&
}
\end{equation}

\begin{equation} \label{diag.eq015}
\xymatrix{
&k  \ar[dr]^{1} & &\\
\bar{W} = (W,W', \psi') &&k &k \ar[l]_{1} & &k \ar@/_3pc/[llllu]_(.57){0} 
       \ar@/^3pc/[lllld]^{0}  \ar[ll]_{1} 
        &k \ar[l]_(.57){0} \\
&k  \ar[ur]_{1} &&&&&
}
\end{equation}

\begin{equation} \label{diag.eq015.1}
\xymatrix{
&&&&&k  \ar@/^1.pc/[ddr]^(.4){1} & &\\
       &&&&&&&\\
\underline{\bar{W}} = (W',W, \underline{\psi'})  && k    
        &k \ar[l]_{1}  &&&k \ar@/_3pc/[lll]_(.6){1} 
              \ar@/^3pc/[lll]^(.6){1}
               \ar[lll]_(.5){0}
              &k \ar[l]_{1} & &\\
         &&&&&&&\\
&&&&&k  \ar@/_1.pc/[uur]_{1} &&&&&
}
\end{equation}

\vspace{.2cm}

Then $\bar{V}, \,\, \bar{W}$ are birepresentations of $(Q, Q')$, and $\underline{\bar{V}}, \,\, \underline{\bar{W}}$ are  birepresentations of  $(Q', Q)$. 

To compute $Rep_{\!_{(Q,Q')}}(\bar{V},\bar{W})$, consider the following diagram.\\

\begin{equation} \label{diag.eq016}
\xymatrix{
&&&k  \ar@/_2pc/[dddddd]_{[a]} \ar[dr]^{\begin{bmatrix}
       1            \\[0.3em]
       0         \\[0.3em]
       \end{bmatrix}} & &\\
\bar{V}\ar[dddddd]_{} &&&&k^2 \ar@/^1pc/[dddddd]|{[c \,\,\,d]} &k \ar@/^1pc/[dddddd]|{[e]}\ar[l]_{\begin{bmatrix}
       1            \\[0.3em]
       1         \\[0.3em]
       \end{bmatrix}} & &k \ar@/^1pc/[dddddd]|{[l]}\ar[ll]_{1} \ar@/_5pc/[llllu]_{1} 
       \ar@/^5pc/[lllld]|(.35){1} &k \ar[l]_{1} \ar@/^1pc/[dddddd]|{[m]}\\
&&& k  \ar@/_2pc/[dddddd]_{[b]}\ar[ur]_{\begin{bmatrix}
       0            \\[0.3em]
       1         \\[0.3em]
       \end{bmatrix}} && \\
&&&&&&\\
&&&&&&\\
&&&&&&\\
&&&k \ar[dr]^{1} & &\\
\bar{W} &&&&k  &k \ar[l]_{1} & &k \ar[ll]_{1} \ar@/_5pc/[llllu]|(.3){0} 
       \ar@/^5pc/[lllld]^{0} &k \ar[l]_{0}\\
&&&k  \ar[ur]_{1} &&
}
\end{equation}

\vspace{.2cm}

The commuting squares give the relations
\begin{equation} \label{diag.eq017}
a = c, \,\,\,\,  b = d, \,\,\,\,  c = d + e, \,\,\,\, l = 0, \,\,\,\, e = l, \,\,\,\,a = 0, \,\,\,\, b =0 .
\end{equation}
Hence, we obtain $Rep_{\!_{(Q,Q')}}(\bar{V},\bar{W}) \cong k$.\\
We  leave it to the reader to compute  $Rep_{\!_{(Q,Q')}}(\bar{W},\bar{V})$,  $Rep_{\!_{(Q',Q)}}(\underline{\bar{V}},\underline{\bar{W}})$ and $Rep_{\!_{(Q',Q)}}(\underline{\bar{W}},\underline{\bar{V}})$. 
\end{example} 

\vspace{.2cm}

\begin{definition} \label{def.3} 
Let $\bar{V} = (V,V',\psi)$, $\bar{W} = (W,W',\psi')$ be birepresentations of $(Q,Q')$. 
Write $ V = (V_i,\phi_{\alpha})$, $V' = (V'_{i'},\phi'_{\beta})$,  
$W = ((W_i,\mu_{\alpha})$, $W' = (W'_{i'},\mu'_{\beta})$. Then 
\begin{equation} \label{diag.eq018}
\bar{V} \oplus \bar{W} \,\,\ = \,\,\ ((V_i \oplus W_i,\begin{bmatrix}
       \phi_{\alpha} & 0            \\[0.3em]
       0 & \mu_{\alpha}         \\[0.3em] \end{bmatrix}), (V'_{i'} \oplus W'_{i'},  \begin{bmatrix}
       \phi'_{\alpha} & 0            \\[0.3em]
       0 & \mu'_{\alpha}         \\[0.3em] \end{bmatrix})
, \begin{bmatrix}
       \psi^{\alpha}_{\beta} & 0            \\[0.3em]
       0 & \psi'^{\alpha}_{\beta}         \\[0.3em] \end{bmatrix}),
\end{equation}
where $(V_i \oplus W_i,\begin{bmatrix}
       \phi_{\alpha} & 0            \\[0.3em]
       0 & \mu_{\alpha}         \\[0.3em] \end{bmatrix})$, $(V'_{i'} \oplus W'_{i'},  \begin{bmatrix}   \phi'_{\alpha} & 0   \\[0.3em]
       0 & \mu'_{\alpha}   \\[0.3em] \end{bmatrix})$ are the direct sums of $(V_i,\phi_{\alpha}),\, (W_i,\mu_{\alpha})$ and $(V'_{i'}, \phi'_{\beta}) ,\, (W'_{i'}, \mu'_{\beta})$ in $Rep_k(Q)$, $Rep_k(Q')$ respectively, is a  birepresentation of $(Q,Q')$ called the \textbf{direct sum} of  $\bar{V}, \,\  \bar{W}$ (in $Rep_{\!_{(Q,Q')}}$). \\

Similarly, direct sums in $Rep_{\!_{(Q_1,Q_2,...,Q_n)}}$ can be defined. \\
\end{definition}

\begin{example} \label{ex.3}

Consider the birepresentations in Example \ref{ex.2}. Then the direct sum $\bar{V} \oplus \bar{W}$ is the birepresentation 
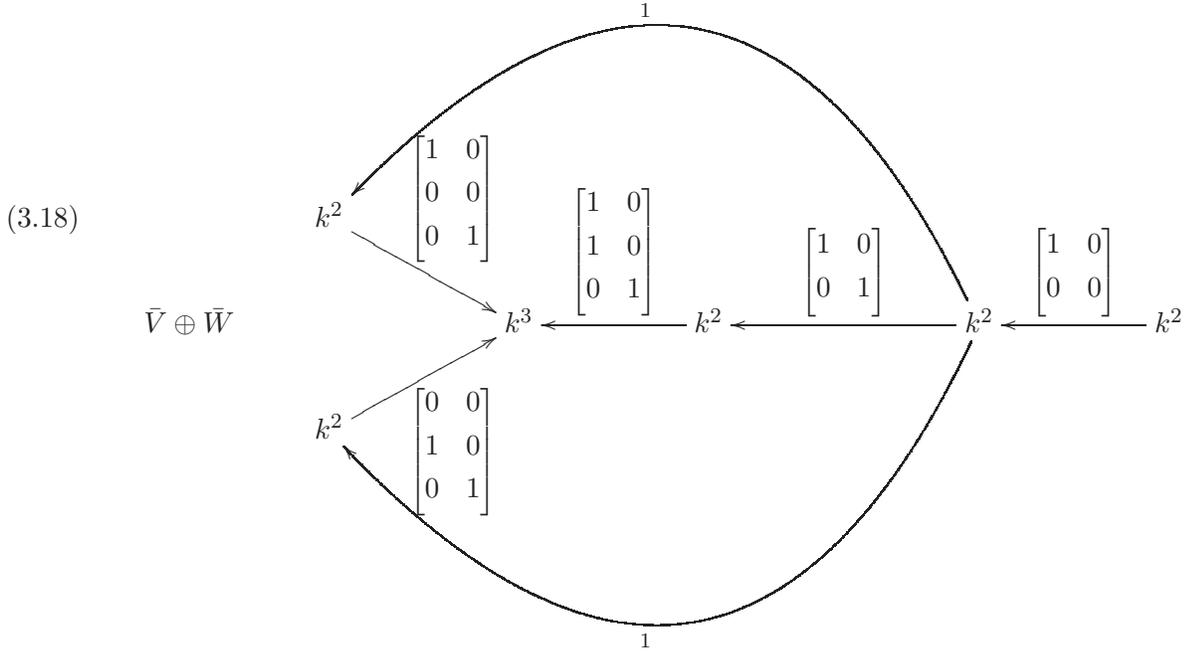
\begin{equation} \label{diag.eq019}
\xymatrix{
&k^2  \ar[drr]^{\begin{bmatrix}
       1 & 0           \\[0.3em]
       0 & 0        \\[0.3em]
       0 & 1            \\[0.3em]
       \end{bmatrix}} & &\\
\bar{V}\oplus \bar{W}  &&&k^3 &&k^2 \ar[ll]_{\begin{bmatrix}
       1 & 0           \\[0.3em]
       1 & 0        \\[0.3em]
       0 & 1            \\[0.3em]
       \end{bmatrix}} & &&k^2 \ar[lll]_{\begin{bmatrix}
       1 & 0           \\[0.3em]
       0 & 1        \\[0.3em]
             \end{bmatrix}} \ar@/_7.8pc/[lllllllu]_(.57){1}  \ar@/^7.8pc/[llllllld]^(.57){1}
        &&k^2 \ar[ll]_{\begin{bmatrix}
       1 & 0           \\[0.3em]
       0 & 0        \\[0.3em]
             \end{bmatrix}} \\
&k^2  \ar[urr]_{\begin{bmatrix}
       0 & 0           \\[0.3em]
       1 & 0        \\[0.3em]
       0 & 1            \\[0.3em]
       \end{bmatrix}} &&&&&
}
\end{equation}\\

\end{example}

\begin{definition} \label{def.4} 
A birepresentation $\bar{V} \in Rep_{\!_{(Q,Q')}}$ is called \textbf{indecomposable} if $\bar{M} \neq 0$ and $\bar{M}$ cannot be written as a direct sum of two nonzero birepresentations, that is,
whenever $\bar{M} \cong \bar{L} \oplus \bar{N}$ with $ \bar{L}, \bar{N} \in Rep_{\!_{(Q,Q')}}$, then $\bar{L} = 0 $ or $\bar{N} = 0$.

\end{definition} 

\begin{example} \label{ex.4}
Consider the birepresentations in Example \ref{ex.2}. The birepresentation $\bar{V}$ is indecomposable, but the  birepresentation $\bar{W}$ is not. 
\end{example} 

\vspace{.2cm}

The above example also shows that if  $\bar{W} = ((W_i,\phi_{\alpha}),(W'_{i'},\phi'_{\beta}), (\psi^{\alpha}_{\beta}))$ is birepresentation of $(Q,Q')$ such that the representations $(W_i,\phi_{\alpha})$ and $(W'_{i'},\phi'_{\beta})$ are indecomposable in $ Rep_k(Q)$, $ Rep_k(Q')$ respectively, then $\bar{W}$ need not be indecomposable. The proof of the following proposition is straightforward.

\begin{proposition} \label{p.2} 
Let $\bar{V}  = ((V_i,\phi_{\alpha}),(V'_{i'},\phi'_{\beta}), (\psi^{\alpha}_{\beta})) \in Rep_{\!_{(Q,Q')}}$ be an indecomposable birepresentation, then the representations $(W_i,\phi_{\alpha})$ and $(W'_{i'},\phi'_{\beta})$ are indecomposable in $ Rep_k(Q)$, $ Rep_k(Q')$ respectively.
\end{proposition}

\begin{proof}

\end{proof}

\vspace{.2cm}

\section{\textbf{Construction For Kernels and Cokernels in $Rep_{\!_{(Q,Q')}}$}}\label{s.construction.kernels.cokernels} 

Following \cite[p. 49]{Wisbauer}, let $\Cc$ be a category with zero object and $f : A \rightarrow B$ a morphism in $\Cc$.
\begin{enumerate}[label=(\roman*)]
\item A morphism $i : K \rightarrow A$ is called a \textbf{kernel} of $f$ if $if = 0$ and, for every morphism $g : D \rightarrow A$ with $gf = 0$, there is a unique morphism $h : D \rightarrow K$ with $hi = g$, i.e. the  diagram is commutative.
\begin{equation} \label{diag.eq019.1}
\xymatrix{
& & D \ar@{-->}[ddll]_{h} \ar[dd]^{g} && \\
&&&&\\
K \ar[rr]_{i} && A \ar[rr]_{f} && B
} 
\end{equation}\\
\item A morphism $p : B \rightarrow C$ is called a \textbf{cokernel} of $f$ if $fp = 0$ and, for every $g : B \rightarrow D$ with $fg = 0$, there is a unique morphism $h : C \rightarrow D$ with $ph = g$, i.e. the  diagram is commutative.
\begin{equation} \label{diag.eq019.2}
\xymatrix{
A \ar[rr]_{f} && B \ar[dd]_{g} \ar[rr]_{p} && C \ar@{-->}[ddll]^{h} \\
&&&&\\
& & D & & 
} 
\end{equation}\\
\end{enumerate}

\begin{remark} \cite[p. 50]{Wisbauer} \label{r.2}
Let $R$ be a  ring with unity and $R-Mod$ (resp $Mod-R$) be category of left $R$-modules (resp right $R$-modules), and let  $f : M \rightarrow N$ be a homomorphism in $R-Mod$ (resp $Mod-R$). Then
\begin{enumerate}[label=(\roman*)]
\item The inclusion $i : Ker \,\ f \rightarrow M $ is a kernel of $f$ in $R-Mod$ (resp $Mod-R$) because if $g : L \rightarrow M $ is given with $fg = 0$, then $g(L) \subseteq Ker \,\ f$. Clearly, we can define   $g' : L \rightarrow Ker \,\ f, \,\ x \mapsto g(x)$ to be the unique morphiism with $ig'=g$.
\item The projection $p : N \rightarrow Coker \,\ f = N/f(M)$ is a cokernel of $f$ in $R-Mod$ (resp $Mod-R$) because if $h : N \rightarrow L $ is given with $hf = 0$, then $ Im \,\ f \subseteq Ker \,\ h$. Thus, by using the First Isomorphism Theorem, we can define  $h' : Coker \,\ f \rightarrow L, \,\ n + f(M) \mapsto h(n)$ to be the unique morphiism with $h'p=h$.\\
\end{enumerate}
\end{remark}

Let $\bar{f}=(f,f'): \bar{V} = (V,V',\psi) \rightarrow \bar{W}  = (W,W',\psi')$ be  a morphism in $ Rep_{\!_{(Q,Q')}}$.  From  Proposition (\ref{p.1}), the kernels of $f,f'$ exist in  $ Rep_k(Q)$, $ Rep_k(Q')$ respectively.  Let $(\kappa,\zeta), \,\ (\kappa',\zeta')$ be the kernels of $f,f'$ in  $ Rep_k(Q)$, $ Rep_k(Q')$ respectively. Write  $\kappa=(\kappa_i,\chi_{\alpha}),\,\ \zeta = (\zeta_{\alpha}), \,\ \kappa'=(\kappa'_{i'},\nu_{\beta}),\,\ \zeta' = (\zeta'_{\beta})$ and consider the following  diagram.

\begin{equation} \label{diag.eq020}
\xymatrix{
\kappa_{s(\alpha)} \ar[rr]^{\chi_{\alpha}} \ar[dr]_{\zeta_{s(\alpha)}} && \kappa_{t(\alpha)} \ar@{-->}[ddd]|!{[dd];[d]}\hole|!{[ddd];[dd]}\hole^(.5){\xi^{\alpha}_{\beta}}  \ar[dr]^{\zeta_{t(\alpha)}}\\
& V_{s(\alpha)} \ar[rr]|(.3){\phi_{\alpha}}  \ar[dr]_{f_{s(\alpha)}}
&& V_{t(\alpha)} \ar[ddd]|!{[dd];[d]}\hole^(.5){\psi^{\alpha}_{\beta}}  \ar[dr]^{f_{t(\alpha)}}\\
&& W_{s(\alpha)} \ar[rr]|(.3){\phi'_{\alpha}} && W_{t(\alpha)} \ar[ddd]|(.5){\psi'^{\alpha}_{\beta}} \\
&& \kappa'_{s(\beta)} \ar[rr]|\hole|(.65){\nu_{\beta}} \ar[dr]_{\zeta'_{s(\beta)}} &&
\kappa'_{t(\beta}  \ar[dr]^{\zeta'_{t(\beta)}} \\
&&&  V'_{s(\beta)}  \ar[dr]_{f'_{s(\beta)}} \ar[rr]|\hole^(.3){\mu_{\beta}} &&  V'_{t(\beta)}\ar[dr]^{f'_{t(\beta)}}\\
&&&&  W'_{s(\beta)}  \ar[rr]_(.5){\mu'_{\beta}} &&  W'_{t(\beta)}
}
\end{equation}

\vspace{.2cm}

For all $ \alpha \in \mathsf{Q}_{\!_{1}}, \,\ \beta \in \mathsf{Q}'_{\!_{1}}$, define $\xi^{\alpha}_{\beta}: \kappa_{t(\alpha)} \rightarrow \kappa'_{s(\alpha)}$ to be the restriction of $\psi^{\alpha}_{\beta}$. Then $\xi^{\alpha}_{\beta}$ is well defined for all $ \alpha \in \mathsf{Q}_{\!_{1}}, \,\ \beta \in \mathsf{Q}'_{\!_{1}}$ since for all $x \in \kappa_{t(\alpha)}$, we have $f'_{s(\beta)} \psi^{\alpha}_{\beta}(x) = \psi'^{\alpha}_{\beta} f_{t(\alpha)}(x) = 0$. Thus,  $\psi^{\alpha}_{\beta}(x) \in \kappa'_{s(\alpha)}$, and it follows that $\xi^{\alpha}_{\beta}$ is well defined. \\

Let $\bar{\kappa} = (\kappa,\kappa',\xi)$ and $\bar{\zeta} = (\zeta,\zeta')$. Then $\bar{\kappa} \in  Rep_{\!_{(Q,Q')}}$, and we have the following proposition. 

\begin{proposition} \label{p.3.ker} 
The pair $(\bar{\kappa},\bar{\zeta})$ is the kernel of $\bar{f}$ in $ Rep_{\!_{(Q,Q')}}$. 
\end{proposition}

\begin{proof}
Let $ \bar{\lambda} = (\lambda,\lambda'): \bar{N} = (N,N',\Psi) \rightarrow \bar{V}  = (V,V',\psi)$ be a morphism in  $ Rep_{\!_{(Q,Q')}}$ with $ \bar{f} \bar{\lambda} = \bar{0} $, where $\bar{0} $ is the zero object in $ Rep_{\!_{(Q,Q')}}$. Consider the following diagram. \\

\begin{equation} \label{diag.eq021}
\xymatrix{
& \kappa_{s(\alpha)}  \ar[rr]^{\chi_{\alpha}} \ar[dr]_{\zeta_{s(\alpha)}} && \kappa_{t(\alpha)} \ar[ddd]|!{[dd];[d]}\hole|!{[ddd];[dd]}\hole^(.5){\xi^{\alpha}_{\beta}}  \ar[dr]^{\zeta_{t(\alpha)}}\\
 & & V_{s(\alpha)}  \ar[rr]|(.3){\phi_{\alpha}}  \ar[dr]_{f_{s(\alpha)}}
&& V_{t(\alpha)}   \ar[ddd]|!{[dd];[d]}\hole^(.5){\psi^{\alpha}_{\beta}}  \ar[dr]^{f_{t(\alpha)}}\\
N_{s(\alpha)} \ar@/^2pc/[uur]|(.5){\tau_{s(\alpha)}}   \ar[urr]^(.5){\lambda_{s(\alpha)}} \ar[d]^{\delta_{\alpha}} & && W_{s(\alpha)}   \ar[rr]|(.3){\phi'_{\alpha}} && W_{t(\alpha)}  \ar[ddd]|(.5){\psi'^{\alpha}_{\beta}} \\
N_{t(\alpha)} \ar@/^10.6pc/[uuurrr]|(.6){\tau_{t(\alpha)}}
\ar@/_2.7pc/[uurrrr]|(.5){\lambda_{t(\alpha)}} \ar[d]_(.4){\Psi^{\alpha}_{\beta}} &&& \kappa'_{s(\beta)} \ar[rr]|\hole|(.65){\rho_{\beta}} \ar[dr]_{\zeta'_{s(\beta)}} &&
\kappa'_{t(\beta}  \ar[dr]^{\zeta'_{t(\beta)}} \\
N'_{s(\beta)} \ar@/_2pc/[urrr]|(.5){\tau'_{s(\beta)}}  \ar@/_2pc/[rrrr]|(.7){\lambda'_{s(\beta)}}  \ar[d]_{\delta'_{\beta}}  &&&&  V'_{s(\beta)} \ar[dr]_{f'_{s(\beta)}} \ar[rr]|\hole|(.7){\mu_{\beta}} &&  V'_{t(\beta)}  \ar[dr]^{f'_{t(\beta)}}\\
N'_{t(\beta)} \ar@/_7.3pc/[urrrrrr]|(.5){\lambda'_{t(\beta)}} 
\ar@/_4.6pc/[uurrrrr]|(.5){\tau'_{t(\beta)}} &&&&&  W'_{s(\beta)}  \ar[rr]_(.5){\mu'_{\beta}} &&  W'_{t(\beta)} 
}
\end{equation}\\

Since $(\kappa_i,\chi_{\alpha}),(\kappa'_{i'},\nu_{\beta})$ are the kernels of $f,f'$ in  $ Rep_k(Q)$, $ Rep_k(Q')$ respectively, there exist unique morphisms $\tau : N \rightarrow \kappa, \,\, \tau' : N' \rightarrow \kappa'$ in  $ Rep_k(Q)$, $ Rep_k(Q')$ respectively making their respective subdiagrams commutative. So all we need is to show that $\bar{\tau}= (\tau,\tau'): \bar{N} \rightarrow \bar{\kappa}$ is a  morphism in  $ Rep_{\!_{(Q,Q')}}$. To check this, the constructions of the kernels of $f,f'$ in  $ Rep_k(Q)$, $ Rep_k(Q')$   respectively  and Remark  (\ref{r.2}) imply that $\tau (\underline{x}) = \lambda(\underline{x})$ and  $\tau'(\underline{x}') = \lambda'(\underline{x}')$ for all $\underline{x} \in N$ and $\underline{x}' \in N'$ respectively. For any $x \in N_{t(\alpha)}$,  $ \alpha \in \mathsf{Q}_{\!_{1}}, \,\ \beta \in \mathsf{Q}'_{\!_{1}}$, we have \\

\begin{tabular}{lllll}
$\tau'_{s(\beta)} \Psi^{\alpha}_{\beta}(x)$  &  $= \tau'_{s(\beta)}( \Psi^{\alpha}_{\beta}(x)) $\\
  & $=  \lambda'_{s(\beta)} (\Psi^{\alpha}_{\beta}(x)) $\\
  & (since $ \lambda'(\underline{x}') = \tau' (\underline{x}')$, for all $\underline{x}' \in N'$)\\
 & $=  \lambda'_{s(\beta)} \Psi^{\alpha}_{\beta} (x) $\\
 & $=  \psi^{\alpha}_{\beta} \lambda_{t(\alpha)}  (x) $\\
&  (since $\bar{\lambda}$ is a morphism in $ Rep_{\!_{(Q,Q')}}$)\\
  \end{tabular}\\
  
  \begin{tabular}{lllll}
\,\,\,\,\,\,\,\,\,\,\,\,\,\,\,\,\,\,\,\,\,\,\,\,\,\,\,\, & $=  \psi^{\alpha}_{\beta} \tau_{t(\alpha)}  (x) $\\
 & (since $ \lambda(\underline{x}) = \tau (\underline{x})$, for all $\underline{x} \in N$)\\
 & $=  \psi^{\alpha}_{\beta} (\tau_{t(\alpha)}  (x)) $\\
 & $=  \xi^{\alpha}_{\beta} (\tau_{t(\alpha)} (x)) $\\
  &  (by the definition of $\xi^{\alpha}_{\beta}$)\\
 & $=  \xi^{\alpha}_{\beta} \tau_{t(\alpha)} (x) $\\
  \end{tabular}\\
  
  \vspace{.2cm}
  
Therefore,  $\bar{\tau}= (\tau,\tau'): \bar{N} \rightarrow \bar{\kappa}$ is a  morphism in  $ Rep_{\!_{(Q,Q')}}$, and thus $(\bar{\kappa},\bar{\zeta})$ is the kernel of $\bar{f}$ in $ Rep_{\!_{(Q,Q')}}$.\\
 
\end{proof}

\vspace{.2cm}

Using induction on $n$ and the same procedure used above, we can show that kernels exist in the category $Rep_{\!_{(Q_1,Q_2,...,Q_n)}}$, and they can similarly be constructed.\\ Explicitly, let $\bar{V} = (V^{(1)},V^{(2)},...,V^{(n)},\psi_{\!_{1}}, \psi_{\!_{2}}, ..., \psi_{\!_{n-1}})$, $\bar{W} = (W^{(1)},W^{(2)},...,W^{(n)},\psi'_{\!_{1}}, \psi'_{\!_{2}}, ..., \psi'_{\!_{n-1}})$ be $n$-representations of $(Q_1, \,\, Q_2, \,\,... \,\, ,  Q_n)$, and let $\underline{f}: \bar{V} \rightarrow \bar{W}$ be a morphism of $n$-representations, where $\underline{f}=(f^{\!^{(1)}}, f^{\!^{(2)}}, ..., f^{\!^{(n)}})$. For any  $m \in \{2,\,\,..., \,\,n\}$, write 
\begin{center}
•$f^{\!^{(m)}} = (f^{\!^{(m)}}_{i^{(m)}}):  (V_{{i^{(m)}}},\phi^{{i^{(m)}}}_{\gamma^{(m)}}) \rightarrow  (W_{{i^{(m)}}},\mu^{{i^{(m)}}}_{\gamma^{(m)}})$,
\end{center}

By induction, the kernel of $(f^{\!^{(1)}}, f^{\!^{(2)}}, ..., f^{\!^{(n-1)}})$ in $Rep_{\!_{(Q_1,Q_2,...,Q_{n-1})}}$ exists in $Rep_{\!_{(Q_1,Q_2,...,Q_{n-1})}}$. Let 
\begin{center}
•$((\kappa^{(1)},\kappa^{(2)},...,\kappa^{(n-1)},\xi_{\!_{1}}, \xi_{\!_{2}}, ..., \xi_{\!_{n-2}}),(\zeta^{\!^{(1)}}, \zeta^{\!^{(2)}}, ..., \zeta^{\!^{(n-1)}}))$
\end{center}
 be the kernel of $(f^{\!^{(1)}}, f^{\!^{(2)}}, ..., f^{\!^{(n-1)}})$ in $Rep_{\!_{(Q_1,Q_2,...,Q_{n-1})}}$. 

For all $(\gamma^{(n-1)},\gamma^{(n)}) \in \mathsf{Q}^{(n-1)}_{\!_{1}} \times  \mathsf{Q}^{(n)}_{\!_{1}}$, define 
\begin{center}
•$\xi_{\!_{n}\gamma^{(n-1)}}^{\gamma^{(n)}}: \kappa^{(n-1)}_{t^{(n-1)}(\gamma^{(n-1)})} \rightarrow \kappa^{(n)}_{s^{(n)}(\gamma^{(n)})} $
\end{center}
to be the restriction of $\xi_{\!_{n-1}\gamma^{(n-2)}}^{\gamma^{(n-1)}}$. Then  $\xi_{\!_{n}\gamma^{(n)}}^{\gamma^{(n-1)}}$ is well defined for all $(\gamma^{(n-1)},\gamma^{(n)}) \in \mathsf{Q}^{(n-1)}_{\!_{1}} \times  \mathsf{Q}^{(n)}_{\!_{1}}$ by using similar argument used for the case $n=2$. \\
Let $\underline{\kappa} = (\kappa^{(1)},\kappa^{(2)},...,\kappa^{(n)},\xi_{\!_{1}}, \xi_{\!_{2}}, ..., \xi_{\!_{n-1}})$,  $\underline{\zeta} = (\zeta^{\!^{(1)}}, \zeta^{\!^{(2)}}, ..., \zeta^{\!^{(n)}})$ and $\underline{f} = (f^{\!^{(1)}}, f^{\!^{(2)}}, ..., f^{\!^{(n)}})$, where $(\kappa^{(n)},\zeta^{\!^{(n)}})$ is the kernel of $f^{\!^{(n)}}$\,\,. Using similar argument used in Proposition (\ref{p.3.ker}) gives the following proposition. \\

\begin{proposition} \label{p.3.1.ker.n.rep} 
The pair $(\underline{\kappa}, \underline{\zeta}))$ is the kernel of $\underline{f}$ in $Rep_{\!_{(Q_1,Q_2,...,Q_{n})}}$.  \\
\end{proposition}

\begin{proof}

\end{proof}

\vspace{.4cm}

Let $\bar{f}=(f,f'): \bar{V} = (V,V',\psi) \rightarrow \bar{W}  = (W,W',\psi')$ be  a morphism in $ Rep_{\!_{(Q,Q')}}$.  From Proposition (\ref{p.1}), the cokernels of $f,f'$ exist in  $ Rep_k(Q)$, $ Rep_k(Q')$ respectively. 
Let $(K,\eta), \,\ (K',\eta')$ be the cokernels of $f,f'$ in  $ Rep_k(Q)$, $ Rep_k(Q')$ respectively. Write  $K=(K_i,\phi'_{\alpha}),\,\ \eta = (\eta_{\alpha}), \,\ K'=(K_i,\mu'_{\beta}),\,\ \eta' = (\eta'_{\beta})$ and Consider the following  diagram.\\

\begin{equation} \label{diag.eq022}
\xymatrix{
V_{s(\alpha)} \ar[rr]^{\phi_{\alpha}} \ar[dr]_{f_{s(\alpha)}} && V_{t(\alpha)} \ar[ddd]|!{[dd];[d]}\hole|!{[ddd];[dd]}\hole^(.5){\psi^{\alpha}_{\beta}}  \ar[dr]^{f_{t(\alpha)}}\\
& W_{s(\alpha)} \ar[rr]|(.3){\phi'_{\alpha}}  \ar[dr]_{\eta_{s(\alpha)}}
&& W_{t(\alpha)} \ar[ddd]|!{[dd];[d]}\hole^(.5){\psi'^{\alpha}_{\beta}}  \ar[dr]^{\eta_{t(\alpha)}}\\
&& K_{s(\alpha)} \ar[rr]|(.3){\phi''_{\alpha}} && K_{t(\alpha)} \ar@{-->}[ddd]|(.5){\psi''^{\alpha}_{\beta}} \\
&& V'_{s(\beta)} \ar[rr]|\hole|(.65){\mu_{\beta}} \ar[dr]_{f'_{s(\beta)}} &&
V'_{t(\beta}  \ar[dr]^{f'_{t(\beta)}} \\
&&&  W'_{s(\beta)}  \ar[dr]_{\eta'_{s(\beta)}} \ar[rr]|\hole^(.3){\mu'_{\beta}} &&  W'_{t(\beta)}\ar[dr]^{\eta'_{t(\beta)}}\\
&&&&  K'_{s(\beta)}  \ar[rr]_(.5){\mu''_{\beta}} &&  K'_{t(\beta)}
}
\end{equation}

For all $ \alpha \in \mathsf{Q}_{\!_{1}}, \,\ \beta \in \mathsf{Q}'_{\!_{1}}$, define the $k$-linear map $\psi''^{\alpha}_{\beta}: K_{t(\alpha)} \rightarrow K'_{s(\alpha)}$ by $\psi''^{\alpha}_{\beta}(w + f_{t(\alpha)}(V_{t(\alpha)})) = \psi'^{\alpha}_{\beta} (w) + f'_{s(\beta)}(V'_{s(\beta)})$ for all $w \in W_{t(\alpha)}$. Then $\psi''^{\alpha}_{\beta}$ is well defined for all $ \alpha \in \mathsf{Q}_{\!_{1}}, \,\ \beta \in \mathsf{Q}'_{\!_{1}}$ since for all $a, b   \in W_{t(\alpha)}$ with  $a + f_{t(\alpha)} (V_{t(\alpha)}) = b + f_{t(\alpha)} (V_{t(\alpha)}) $, we have $a - b \in  f_{t(\alpha)} (V_{t(\alpha)}) $. Thus $\psi'^{\alpha}_{\beta} (a) - \psi'^{\alpha}_{\beta} (b)\,\  = \,\  \psi'^{\alpha}_{\beta} (a - b)  \in  \psi'^{\alpha}_{\beta} f_{t(\alpha)} (V_{t(\alpha)}) \,\  = \,\ f'_{s(\beta)} \psi^{\alpha}_{\beta} (V_{t(\alpha)}) \subseteq  \,\ f'_{s(\beta)} (V'_{s(\alpha)})$. It follows that $\psi'^{\alpha}_{\beta} (a) + f'_{s(\beta)}(V'_{s(\alpha)}) = \psi'^{\alpha}_{\beta} (b) + f'_{s(\beta)}(V'_{s(\alpha)})$, and hence $\psi''^{\alpha}_{\beta}$ is well defined. \\

Let $\bar{K} = (K,K',\psi'')$ and $\bar{\eta} = (\eta,\eta')$. Then $\bar{K} \in  Rep_{\!_{(Q,Q')}}$ and We have the following proposition. 

\begin{proposition} \label{p.4.coker} 
The pair $(\bar{K},\bar{\eta})$ is the cokernel of $\bar{f}$. 
\end{proposition}

\begin{proof}
Let $ \bar{\gamma}: \bar{W}  = (W,W',\psi) \rightarrow \bar{L}  = (L,L',\Psi)$ be a morphism in  $ Rep_{\!_{(Q,Q')}}$ with $\bar{\gamma} \bar{f} = \bar{0} $ and consider the following diagram. \\

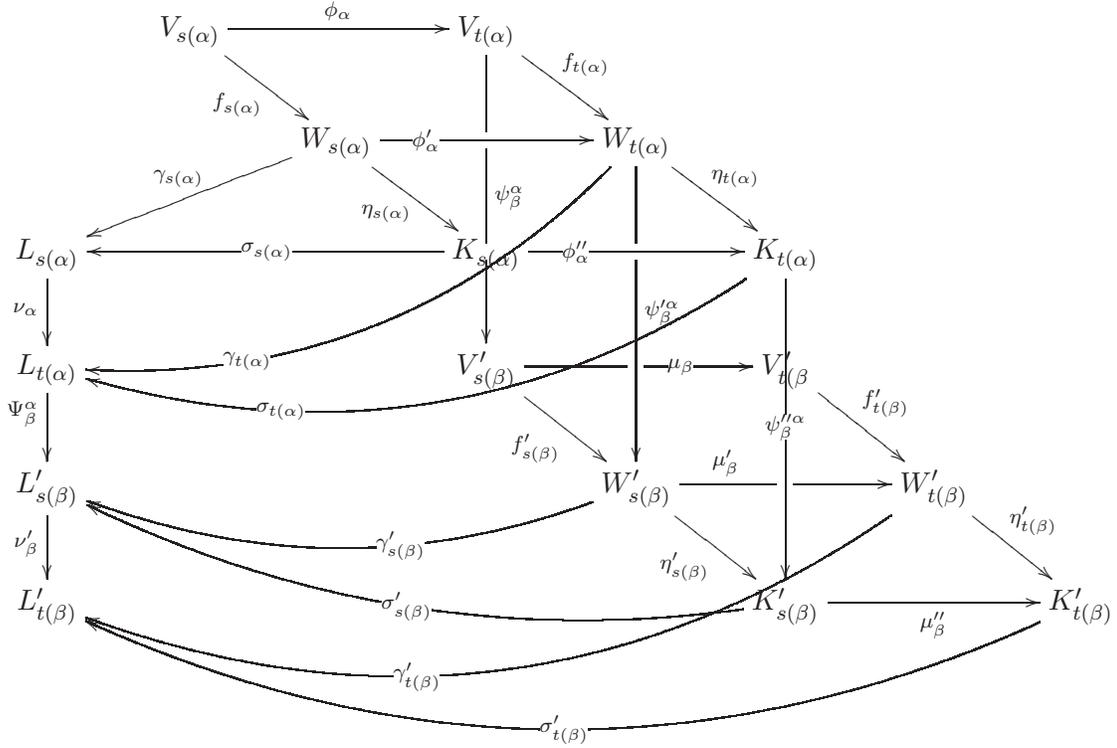
\begin{equation} \label{diag.eq023}
\xymatrix{
& V_{s(\alpha)} \ar[rr]^{\phi_{\alpha}} \ar[dr]_{f_{s(\alpha)}} && V_{t(\alpha)} \ar[ddd]|!{[dd];[d]}\hole|!{[ddd];[dd]}\hole^(.5){\psi^{\alpha}_{\beta}}  \ar[dr]^{f_{t(\alpha)}}\\
 & & W_{s(\alpha)} \ar[dll]_(.5){\gamma_{s(\alpha)}} \ar[rr]|(.3){\phi'_{\alpha}}  \ar[dr]_{\eta_{s(\alpha)}}
&& W_{t(\alpha)}  \ar@/^2.5pc/[ddllll]|(.7){\gamma_{t(\alpha)}} \ar[ddd]|!{[dd];[d]}\hole^(.5){\psi'^{\alpha}_{\beta}}  \ar[dr]^{\eta_{t(\alpha)}}\\
 L_{s(\alpha)}   \ar[d]_{\nu_{\alpha}} & && K_{s(\alpha)} \ar[lll]|(.5){\sigma_{s(\alpha)}}   \ar[rr]|(.3){\phi''_{\alpha}} && K_{t(\alpha)} \ar@/^3pc/[dlllll]|(.7){\sigma_{t(\alpha)}} \ar[ddd]|(.5){\psi''^{\alpha}_{\beta}} \\
 L_{t(\alpha)} \ar[d]_(.4){\Psi^{\alpha}_{\beta}} &&& V'_{s(\beta)} \ar[rr]|\hole|(.65){\mu_{\beta}} \ar[dr]_{f'_{s(\beta)}} &&
V'_{t(\beta}  \ar[dr]^{f'_{t(\beta)}} \\
L'_{s(\beta)}   \ar[d]_{\nu'_{\beta}}  &&&&  W'_{s(\beta)} \ar@/^2pc/[llll]|(.4){\gamma'_{s(\beta)}} \ar[dr]_{\eta'_{s(\beta)}} \ar[rr]|\hole^(.3){\mu'_{\beta}} &&  W'_{t(\beta)} \ar@/^4pc/[dllllll]|(.6){\gamma'_{t(\beta)}} \ar[dr]^{\eta'_{t(\beta)}}\\
 L'_{t(\beta)} &&&&&  K'_{s(\beta)} \ar@/^2pc/[ulllll]|(.5){\sigma'_{s(\beta)}}  \ar[rr]_(.5){\mu''_{\beta}} &&  K'_{t(\beta)} \ar@/^4pc/[lllllll]|(.5){\sigma'_{t(\beta)}}
}
\end{equation}

\vspace{.2cm}

Since $(K_i,\phi_{\alpha}),(K'_{i'},\mu_{\beta})$ are the cokernels of $f,f'$ in  $ Rep_k(Q)$, $ Rep_k(Q')$ respectively, there exist unique morphisms $\sigma : K \rightarrow L, \,\, \sigma' : K' \rightarrow L'$ in  $ Rep_k(Q)$, $ Rep_k(Q')$ respectively making their respective diagrams commutative. So all we need is to show that $\bar{\sigma}= (\sigma,\sigma'): \bar{K} \rightarrow \bar{L}$ is a  morphism in  $ Rep_{\!_{(Q,Q')}}$. To show this, the constructions of the  cokernels of $f,f'$ in  $ Rep_k(Q)$, $ Rep_k(Q')$,  respectively, and Remark  (\ref{r.2}) imply that $\sigma (x+f(V)) = \gamma(x)$ and $\sigma' ((x'+f'(V'))) = \gamma'(x')$ for all $x \in W$ and $x' \in W'$ respectively. For any $w \in W_{t(\alpha)}$, we have \\

\begin{tabular}{lllll}
$\sigma'_{s(\beta)} \psi''^{\alpha}_{\beta}(w + f_{t(\alpha)}(V_{t(\alpha)}))$  &  $= \sigma'_{s(\beta)} (\psi'^{\alpha}_{\beta} (w) + f'_{s(\beta)}(V'_{s(\beta)})) $\\
&  (by the definition of $\psi''^{\alpha}_{\beta}$)\\
 & $= \gamma'_{s(\beta)} \psi'^{\alpha}_{\beta} (w)$\\
&  (since $\sigma' ((x'+f'(V'))) = \gamma'(x')$ for all $x' \in W'$) \\
 & $= \Psi^{\alpha}_{\beta} \gamma_{t(\alpha)} (w)  $\\
&  (since $\bar{\gamma}$ is a morphism in $ Rep_{\!_{(Q,Q')}}$) \\
& $= \Psi^{\alpha}_{\beta} \sigma_{t(\alpha)} \eta_{s(\alpha)} (w)$ \\
& (since $\bar{\gamma}= \bar{\sigma} \bar{\eta}$) \\
& $= \Psi^{\alpha}_{\beta} \sigma_{t(\alpha)} (w + f_{t(\alpha)}(V_{t(\alpha)}))$ \\
& (by the definition of $\eta$)\\
  \end{tabular}\\

\vspace{.2cm}
  
Therefore, $\bar{\sigma}: \bar{K} \rightarrow \bar{L}$ is a  morphism in  $ Rep_{\!_{(Q,Q')}}$, and thus $(\bar{K},\bar{\eta})$ is the cokernel of $\bar{f}$. \\

\end{proof}

\vspace{.2cm}

Using induction on $n$, one can use the same procedure used above to show that cokernels exist in the category $Rep_{\!_{(Q_1,Q_2,...,Q_n)}}$, and they can similarly be constructed.\\ Explicitly, let $\bar{V} = (V^{(1)},V^{(2)},...,V^{(n)},\psi_{\!_{1}}, \psi_{\!_{2}}, ..., \psi_{\!_{n-1}})$, $\bar{W} = (W^{(1)},W^{(2)},...,W^{(n)},\psi'_{\!_{1}}, \psi'_{\!_{2}}, ..., \psi'_{\!_{n-1}})$ be $n$-representations of $(Q_1, \,\, Q_2, \,\,... \,\, ,  Q_n)$, and let $\underline{f}: \bar{V} \rightarrow \bar{W}$ be a morphism of $n$-representations, where $\underline{f}=(f^{\!^{(1)}}, f^{\!^{(2)}}, ..., f^{\!^{(n)}})$. For any  $m \in \{2,\,\,..., \,\,n\}$, write 
\begin{center}
•$f^{\!^{(m)}} = (f^{\!^{(m)}}_{i^{(m)}}):  (V_{{i^{(m)}}},\phi^{{i^{(m)}}}_{\gamma^{(m)}}) \rightarrow  (W_{{i^{(m)}}},\mu^{{i^{(m)}}}_{\gamma^{(m)}})$,
\end{center}

By induction, the cokernel of $(f^{\!^{(1)}}, f^{\!^{(2)}}, ..., f^{\!^{(n-1)}})$ in $Rep_{\!_{(Q_1,Q_2,...,Q_{n-1})}}$ exists in $Rep_{\!_{(Q_1,Q_2,...,Q_{n-1})}}$. Let 
\begin{center}
•$((K^{(1)},K^{(2)},...,K^{(n-1)},\chi_{\!_{1}}, \chi_{\!_{2}}, ..., \chi_{\!_{n-2}}),(\eta^{\!^{(1)}}, \eta^{\!^{(2)}}, ..., \eta^{\!^{(n-1)}}))$
\end{center}
 be the cokernel of $(f^{\!^{(1)}}, f^{\!^{(2)}}, ..., f^{\!^{(n-1)}})$ in $Rep_{\!_{(Q_1,Q_2,...,Q_{n-1})}}$. 

For all $(\gamma^{(n-1)},\gamma^{(n)}) \in \mathsf{Q}^{(n-1)}_{\!_{1}} \times  \mathsf{Q}^{(n)}_{\!_{1}}$, define 

 \begin{center}
•$\chi_{\!_{n}\gamma^{(n-1)}}^{\gamma^{(n)}}: K^{(n-1)}_{t^{(n-1)}(\gamma^{(n-1)})} \rightarrow K^{(n)}_{s^{(n)}(\gamma^{(n)})} $
\end{center}
by $\chi_{\!_{n}\gamma^{(n-1)}}^{\gamma^{(n)}}(w + f^{(n-1)}_{t^{(n-1)}(\gamma^{(n-1)})}(V^{(n-1)}_{t^{(n-1)}(\gamma^{(n-1)})})) = {\psi'}_{\!_{n}\gamma^{(n-1)}}^{\gamma^{(n)}} (w) + f^{(n)}_{t^{(n)}(\gamma^{(n)})}(V^{(n)}_{t^{(n)}(\gamma^{(n)})})$ for all $w \in W^{(n-1)}_{t^{(n-1)}(\gamma^{(n-1)})}$. Then  $\chi_{\!_{n}\gamma^{(n-1)}}^{\gamma^{(n)}}$ is well defined for all $(\gamma^{(n-1)},\gamma^{(n)}) \in \mathsf{Q}^{(n-1)}_{\!_{1}} \times  \mathsf{Q}^{(n)}_{\!_{1}}$ by using similar argument used for the case $n=2$. \\
Let $\underline{K} = (K^{(1)},K^{(2)},...,K^{(n)},\chi_{\!_{1}}, \chi_{\!_{2}}, ..., \chi_{\!_{n-1}})$,  $\underline{\eta} = (\eta^{\!^{(1)}}, \eta^{\!^{(2)}}, ..., \eta^{\!^{(n)}})$ and $\underline{f} = (f^{\!^{(1)}}, f^{\!^{(2)}}, ..., f^{\!^{(n)}})$, where $(K^{(n)},\eta^{\!^{(n)}})$. Using similar argument used in Proposition (\ref{p.4.coker}) gives the following consequence. \\

\begin{proposition} \label{p.4.1.coker.n.rep} 
The pair $(\underline{K}, \underline{\eta}))$ is the cokernel of $\underline{f}$ in $Rep_{\!_{(Q_1,Q_2,...,Q_{n})}}$.  \\
\end{proposition}

\begin{proof}

\end{proof}

\vspace{.2cm}

\section{\textbf{Canonical Decomposition of Morphisms in $Rep_{\!_{(Q,Q')}}$}}\label{s.Factorization} 
Following \cite[p. 2]{Etingof},  an \textbf{additive} category is a category $\Cc$ satisfying the following
axioms:
\begin{enumerate}[label=(\roman*)]
\item Every set $\Cc(X,Y)$ is equipped with a structure of an abelian group (written additively) such that composition of morphisms is biadditive with respect to this structure.
\item There exists a zero object $0 \in \Cc$ such that $\Cc(0,0)=0$.
\item (Existence of direct sums.) For any objects $X,X' \in \Cc$, the direct sum $X \oplus X' \in \Cc$.
\end{enumerate} 
Let $k$ be a field. An additive category $\Cc$ is said to be\textbf{ $k$-linear} if for any objects $X,Y \in \Cc(X,Y)$ is equipped
with a structure of a vector space over $k$, such that composition of morphisms is $k$-linear.\\
An \textbf{abelian} category is an additive category $\Cc$ in which for every morphism $f : X \rightarrow Y$ there exists a sequence 
\begin{equation} \label{diag.eq025}
\xymatrix{
K \ar[r]^{k} & X \ar[r]^{\iota} & I \ar[r]^{j} & Y \ar[r]^{c} & C
} 
\end{equation} \\

with the following properties:
\begin{enumerate}[label=(\roman*)]
\item $ji = f$,
\item $(K,k) = Ker(f), \,\ (C, c) = Coker(f)$,
\item $(I, \iota) = Coker(k), \,\ (I, j) = Ker(c)$.\\
\end{enumerate}  
A sequence (\ref{diag.eq025}) is called a \textbf{canonical decomposition} of $f$. \\
   
Let $\bar{f} =(f,f') : \bar{X} \rightarrow \bar{Y}$ be a morphism in $Rep_{\!_{(Q,Q')}}$.  
It follows that $f : X \rightarrow Y, \,\ f' : X' \rightarrow Y'$ are morphisms in $Rep_k(Q)$, $Rep_k(Q')$ respectively. From Proposition \ref{p.1}, $f : X \rightarrow Y, \,\ f' : X' \rightarrow Y'$ have the canonical decompositions 
\begin{equation} \label{diag.eq026}
\xymatrix{
K \ar[r]^{k} & X \ar[r]^{\iota} & I \ar[r]^{j} & Y \ar[r]^{c} & C
} 
\hspace{70pt}
\xymatrix{
K' \ar[r]^{k'} & X' \ar[r]^{\iota'} & I' \ar[r]^{j'} & Y' \ar[r]^{c'} & C'
} 
\end{equation}
in $Rep_k(Q)$, $Rep_k(Q')$ respectively.  From Section \ref{s.construction.kernels.cokernels}, kernels and cokernels exist in  $Rep_{\!_{(Q,Q')}}$. It turns out that $\bar{f}$ has a canonical decomposition 
\begin{equation} \label{diag.eq025}
\xymatrix{
\bar{K} \ar[r]^{\bar{k}} & \bar{X} \ar[r]^{\bar{\iota}} & \bar{I} \ar[r]^{\bar{j}} & \bar{Y} \ar[r]^{\bar{c}} & \bar{C}
}
\end{equation} 
in $Rep_{\!_{(Q,Q')}}$, and this decomposition can explicitly be seen in the following commutative diagram.
    
\begin{equation} \label{diag.eq028}
\xymatrix{
&K_{s(\alpha)} \ar[rr]^{\phi_{\alpha}} \ar[dr]|{k_{s(\alpha)}} && K_{t(\alpha)} \ar[ddd]|!{[dd];[d]}\hole|!{[ddd];[dd]}\hole^(.5){\psi^{\alpha}_{\beta}}  \ar[dr]^{k_{t(\alpha)}}\\
&& X_{s(\alpha)} \ar[rr]|(.3){\phi'_{\alpha}}  \ar[dr]|{\iota_{s(\alpha)}}
&& X_{t(\alpha)} \ar[ddd]|!{[dd];[d]}\hole^(.5){\psi'^{\alpha}_{\beta}}  \ar[dr]^{\iota_{t(\alpha)}}\\
&C_{s(\alpha)} \ar[dl]|{\varphi''_{\alpha}} &Y_{s(\alpha)} \ar[dl]|{\varphi'_{\alpha}} \ar[l]_{c_{s(\alpha)}}& I_{s(\alpha)} \ar[rr]|(.3){\phi''_{\alpha}}  \ar[l]_{j_{s(\alpha)}} && I_{t(\alpha)}  \ar@/_.35pc/[dllll]|(.7){j_{t(\alpha)}} \ar[ddd]|(.5){\psi''^{\alpha}_{\beta}} \\
C_{t(\alpha)} \ar@/_4pc/[ddr]_(.5){\Psi'^{\alpha}_{\beta}} &Y_{t(\alpha)} \ar@/^1pc/[ddr]^(.5){\Psi^{\alpha}_{\beta}} \ar[l]_{c_{t(\alpha)}}  &  & K'_{s(\beta)} \ar[rr]|\hole|(.65){\mu_{\beta}} \ar[dr]_{k'_{s(\beta)}} & &
K'_{t(\beta}  \ar[dr]^{k'_{t(\beta)}} \\
C'_{t(\beta)} & Y'_{t(\beta)} \ar[l]_(.5){c'_{s(\beta)}} &&&  X'_{s(\beta)}  \ar[dr]_{\iota'_{s(\beta)}} \ar[rr]|\hole^(.3){\mu'_{\beta}} &&  X'_{t(\beta)}\ar[dr]^{\iota'_{t(\beta)}}\\
&C'_{s(\beta)} \ar@/^0.7pc/[ul]|(.5){\nu'_{\beta}} & Y'_{s(\beta)} \ar@/^0.7pc/[ul]|(.5){\nu_{\beta}} \ar@/^1pc/[l]^(.5){c'_{t(\beta)}}  & &&  I'_{s(\beta)} \ar@/^3.5pc/[lll]|(.5){j'_{t(\beta)}}  \ar[rr]_(.5){\mu''_{\beta}} &&  I'_{t(\beta)} \ar@/^3.5pc/[ullllll]|(.5){j'_{t(\beta)}}
}
\end{equation}  

\vspace{.2cm}

This implies that any  morphism $\underline{f}: \bar{V} \rightarrow \bar{W}$ of $n$-representations has a canonical decomposition in $Rep_{\!_{(Q_1,Q_2,...,Q_{n})}}$.\\

\begin{remark} 
Let $\,\, \bar{f}, \,\, \bar{g}: \bar{V} \rightarrow \bar{W}$ be morphisms in  $Rep_{\!_{(Q,Q')}}$. Write   $\bar{f} = (f,f'), \,\, \bar{g} = (g,g')$,  $f=(f_i), \,\, g=(g_i)$, $f'=(f'_{i'}), \,\, g'=(g'_{i'})$, $\bar{V} = (V,V',\psi)$, $\bar{W} = (W,W',\psi')$. Define $\bar{f} + \bar{g} = (f+g,f'+g') = ((f_i + g_i), (f'_{i'} + g'_{i'}))$. Since $Rep_k(Q)$ and $Rep_k(Q')$ are abelian, the sets  $Rep_k(Q)(V,W)$, $Rep_k(Q)(V',W')$ are equipped with a structure of an abelian group such that composition of morphisms is biadditive with respect to this structure \cite[p. 70]{Assem}. 
Since $\,\, \bar{f}, \,\, \bar{g}$ are  morphisms in  $Rep_{\!_{(Q,Q')}}$ and since the category $Vec_k$ is abelian, we have the following commutative diagram.

\begin{equation} \label{diag.eq010}
\xymatrix{
V_{s(\alpha)} \ar[rr]^{\phi_{\alpha}} \ar[dr]_{f_{s(\alpha)} + g_{s(\alpha)}} && V_{t(\alpha)} \ar[dd]|\hole^(.3){\psi^{\alpha}_{\beta}} \ar[dr]^{f_{t(\alpha)} + g_{t(\alpha)}}\\
& W_{s(\alpha)} \ar[rr]|(.3){\phi'_{\alpha}} && W_{t(\alpha)} \ar[dd]^(.32){\psi'^{\alpha}_{\beta}}\\
&& V'_{s(\beta)} \ar[rr]|\hole|(.65){\mu_{\beta}} \ar[dr]_{f'_{s(\beta)} + g'_{s(\beta)}} &&
V'_{t(\beta}  \ar[dr]^{f'_{t(\beta)} + g'_{t(\beta)}} \\
&&&  W'_{s(\beta)}  \ar[rr]_{\mu'_{\beta}} &&  W'_{t(\beta)}
}
\end{equation}
Thus, the set $Rep_{\!_{(Q,Q')}}(\bar{V},\bar{W})$ is equipped with a structure of an abelian group such that composition of morphisms is biadditive with respect to the above structure.\\
\end{remark}
 
We end the paper with the following crucial results.\\

\begin{theorem} \label{thm.Abelian1} 
The category  $Rep_{\!_{(Q,Q')}}$ is a $k$-linear abelian category. \\
\end{theorem} 

\begin{proof} 

\end{proof} 
  
\begin{theorem} \label{thm.Abelian2} 
The category  $Rep_{\!_{(Q_1,Q_2,...,Q_{n})}}$ is a $k$-linear abelian category for any integer $n \geq 2$. \\
\end{theorem} 

\begin{proof} 

\end{proof}

 \vspace{.5cm}

\begin{center}
• \textbf{Acknowledgment}
\end{center}
I would like to thank my adivsor Prof. Miodrag Iovanov, who I learned a lot from him,  for his support and his unremitting encouragement.\\

\vspace*{3mm} 
\begin{flushright}
\begin{minipage}{148mm}\sc\footnotesize

Adnan Hashim Abdulwahid\\
University of Iowa, \\
Department of Mathematics, MacLean Hall\\
Iowa City, IA, USA

{\tt \begin{tabular}{lllll}
{\it E--mail address} : &   {\color{blue} adnan-al-khafaji@uiowa.edu}\\
&  {\color{blue} adnanalgebra@gmail.com}\\
\end{tabular} }\vspace*{3mm}
\end{minipage}
\end{flushright}


\vspace{.5cm}


\begin{thebibliography}{}

\bibitem{Adamek} J. Adamek , H. Herrlich,  G. Strecker. {\it Abstract and Concrete Categories: The Joy of Cats}. Dover Publication. 2009.

\bibitem{Assem} Ibrahim Assem, Andrzej Skowronski, Daniel Simson. {\it Elements of the Representation Theory of Associative Algebras}. Volume 1: Techniques of Representation Theory. London Mathematical Society Student Texts 65, Cambridge University Press, Cambridge, 2006.

\bibitem{Auslander1} M. Auslander, Idun Reiten, S.O. Smal{\o}.  {\it Representation Theory of Artin Algebras}. Cambridge Studies in Advanced Mathematics 36, Cambridge University Press, Cambridge. 1995.

\bibitem{Awodey} Steve Awodey. {\it Category Theory}. Oxford University Press, 2nd ed., 2010.

\bibitem{Barot} Michael Barot. {\it Introduction to The Representation Theory of Algebras}.  $\copyright$ Springer International Publishing Switzerland. 2015.

\bibitem{Benson} D. J. Benson. {\it Representations and cohomology I: Basic representation theory of finite groups and associative algebras}.  Cambridge Stud. Adv. Math. 30, Cambridge University Press, 1991.
 
\bibitem{Borceux1} F. Borceux. {\it Handbook of Categorical Algebra 1: Basic Category Theory}. Cambridge University Press. 1994.

\bibitem{Buan}  Aslak Bakke Buan, Idun Reiten, {\o}yvind Solberg. {\it Algebras, Quivers and Representations}.  Springer Science+Business Media B.V.. 2013. ISBN 978-3-642-39484-3. 

\bibitem{Dvalishvili} B. Dvalishvili. {\it Bitopological Spaces: Theory, Relations with Generalized Algebraic Structures and Applications}. Elsevier Amsterdam. 2005.

\bibitem{Etingof} Pavel Etingof, Shlomo Gelaki, Dmitri Nikshych, Victor Ostrik. {\it Tensor Categories}. Mathematical Surveys and Monographs 205, American Mathematical Society, 2015.

\bibitem{Etingof1} Pavel Etingof, Oleg Golberg, Sebastian Hensel
Tiankai Liu, Alex Schwendner, Dmitry Vaintrob, Elena Yudovina. {\it Introduction to Representation Theory}. Volume 59, American Mathematical Society, RI, 2011.

\bibitem{Freyd} P.J. Freyd, A. Scedrov. {\it Categories, Allegories}. Elsevier Science Publishing Company, Inc. 1990.

\bibitem{Gabriel} P. Gabriel. {\it Unzerlegbare Darstellungen I}. Manuscripta Math., 6 (1972), 71–103.


\bibitem{Kelly} Kelly J. C. {\it Bitopological spaces}. Proc. London Math. Soc. (3) 13.(1963), 71-89. 


\bibitem{Leinster1} Tom Leinster. {\it Basic Category Theory}. volume 298 of London Mathematical Society Lecture Note Series. Cambridge University Press, Cambridge, 2014.

\bibitem{Leinster2} Tom Leinster. {\it Higher Operads, Higher Categories}. volume 298 of London Mathematical Society Lecture Note Series. Cambridge University Press, Cambridge, 2004. 

\bibitem{Mac Lane1} S. Mac Lane. {\it Categories for the Working Mathematician}. Graduate Texts in Mathematics vol. 5, Springer-Verlag, New York, 2nd ed., 1998. 

\bibitem{McLarty} Colin McLarty. {\it Elementary Categories, Elementary Toposes}. Oxford University Press. New York. 2005.

\bibitem{Mitchell} B. Mitchell. {\it Theory of Categories}.  Academic Press, 1965.
 
\bibitem{Pareigis} B. Pareigis. {\it Categories and Functors}.  Academic Press, 1971.

\bibitem{Rotman} Joseph J. Rotman. {\it An Introduction to Homological Algebra}. 2nd Edition, Springer, New York, 2009.


\bibitem{Schiffler} Ralf Schiffler. {\it Quiver Representations}. CMS Books in Mathematics Series. Springer International Publishing. 2014.

\bibitem{Schubert} H. Schubert. {\it Categories}. Springer-Verlag, 1972.

\bibitem{Sergeichuk} V.V. Sergeichuk. {\it Linearization method in classification problems of linear algebra}. S$\tilde{a}$o Paulo J. Math. Sci.1(no. 2) (2007), 219–240.


\bibitem{Wisbauer} R. Wisbauer. {\it Foundations of Module and Ring Theory: A Handbook for Study and Research}. Springer-Verlag. 1991.

\bibitem{Zimmermann} Alexander Zimmermann. {\it Representation Theory: A Homological Algebra Point of View}. Springer Verlag London, 2014.

\end{thebibliography}
\end{document}